\documentclass{article}
\usepackage{latexsym}
\usepackage{epsfig}
\usepackage{amsmath,amsthm,amssymb}
\usepackage{psfrag}
\def\cG{{\cal G}}
\newcommand{\brac}[1]{\left(#1\right)}

\def\cP{{\cal P}}

\def\d{\delta}

\def\e{\epsilon}

\def\om{\omega}

\def\Pr{\mbox{{\bf Pr}}}

\def\whp{{\bf whp}}

\newtheorem{lemma}{Lemma}
\newtheorem{theorem}{Theorem}

\newtheorem{remark}{Remark}

\newtheorem{definition}{Definition}
\newtheorem{proposition}{Proposition}

\def\e{\epsilon}

\def\Pr{\mbox{{\bf Pr}}}

\newcommand{\diam}{{\rm diam}}
\newcommand{\dist}{{\rm dist}}

\begin{document}
\renewcommand{\theenumi}{(\alph{enumi})}

\large

\makeatletter \title{Adding random edges to dense graphs} \author{Tom
Bohman\thanks{Department of Mathematical Sciences, Carnegie Mellon
University, Pittsburgh, PA 15213.  E-mail: {\tt
tbohman@moser.math.cmu.edu}. Supported in part by NSF grant
DMS-0100400.} \and \ Alan Frieze\thanks{Department of Mathematical
Sciences, Carnegie Mellon University, Pittsburgh, PA 15213.
E-mail: {\tt
alan@random.math.cmu.edu}.
Supported
in part by NSF grant CCR-9818411.}  \and Michael
Krivelevich\thanks{Department of Mathematics, Sackler Faculty of Exact
Sciences, Tel Aviv University, Tel Aviv 69978, Israel. E-mail: {\tt
krivelev@post.tau.ac.il}. Research supported in part by a USA-Israel
BSF grant, by a grant from the Israel Science Foundation and by a
Bergmann Memorial Grant.}  \and \ Ryan Martin\thanks{Department of
Mathematical Sciences, Carnegie Mellon University, Pittsburgh, PA
15213.  E-mail: {\tt rymartin@andrew.cmu.edu}.  Supported in part by
NSF VIGRE grant DMS-9819950.}}\date{} \maketitle \makeatother
\begin{abstract}
This paper investigates the addition of random edges to
arbitrary dense graphs; in particular, we determine the number
of random edges required to ensure various monotone
properties including the
appearance of a fixed size clique, small diameter and
$k$-connectivity.
\end{abstract}

\section{Introduction}
We consider the following random graph model, which was introduced
in~\cite{BFM}: a fixed
graph $H=(V,E)$ is given and a set \( R \) of \(m\) edges is
chosen uniformly at random from  \( \binom{V}{2} \setminus E \)
(i.e. \(R\) is a subset of the non-edges of \(H\))
in order to form a graph
$$G_{H,m} = ( V, E \cup R).$$
For a fixed constant \( 0 < d < 1 \)
let $\cG(n,d)$ denote the set of
graphs on vertex set $[n]$ which have minimum
degree at least $dn$.  We will refer to these graphs
as {\em dense graphs}.  We consider \(G_{H,m} \) where \(H\)
is an arbitrary graph in $\cG(n,d)$.

%

In the classical model of a random graph (Erd\H{o}s and R\'enyi
\cite{ER}) we add random edges to an empty graph, all at once or one
at a time, and then ask for the probability that certain structures
occur.  This model is well-studied: it has generated a wide variety of
research including at least two excellent texts, Bollob\'as~\cite{B0}
and Janson, {\L}uczak and Ruci\'nski~\cite{JLR}.
There are other random graph models as well.  For example, there is a
well established theory of considering random subgraphs of special
graphs.  In fact, our model can be seen to be an analogue by taking
complements.

The property investigated in~\cite{BFM} was hamiltonicity: How
large must \(m\) be in order to ensure that, with high probability,
\(G_{H,m}\) is hamiltonian (we will say that an event
occurs with high
probability, abbreviated \whp, if the probability of that
event approaches 1 as
$n\rightarrow\infty$)?
It was proved in~\cite{BFM} that there exists a constant
$c$, depending on $d$, such that $m\geq cn$ ensures a
Hamilton cycle in \(G_{H,m}\) \whp.  In addition,
there is a constant \(C\), depending on \(d\), and a graph \(H\)
with minimum degree $dn$ which remains
non-hamiltonian when adding as many as $Cn$ edges.  It is well-known
(see~\cite{B0}, ch. VIII) that the threshold function for the number
of random edges required to induce a Hamilton cycle in an empty graph is
$\left\lfloor (n/2)(\log n+\log\log n+\phi(n))\right\rfloor$ for any
function $\phi(n)\rightarrow\infty$ as $n\rightarrow\infty$.  Thus,
although there are more total edges necessary to ensure the Hamilton
cycle in our model than in $G_{n,M}$, we require fewer random edges.

In this paper we consider the clique number,
chromatic number, diameter and vertex
connectivity of \( G_{H,m} \).
For most of our results, we will be most interested in the case where
$d$ is a small constant.  In the case of a Hamilton cycle, for
example, no random edges are required for $d\geq 1/2$.
Note also that connectivity is one of the first properties to emerge
when adding edges.  Since every connected
component in a dense graph has at
least $ dn $ vertices, adding $\omega(1)$
random edges will cause the
resulting graph to be connected \whp.
Thus, if \(m\) random edges give a certain
monotone property monotone \(P\) {\bf whp} then
\( m + \omega(1) \) random edges edges give {\bf whp} a
graph that is both connected and has property \(P\).

We first look at the number of random edges needed to force $G_{H,m}$ to
contain a copy of \(K_r\) where \(r\) is a constant.  For
an arbitrary fixed graph \(H \) let
\[ m(H)=\max\left\{\frac{e(H')}{v(H')} : H'\subseteq
H, v(H')>0\right\}. \]
For the standard random graph we have the following.
\begin{theorem}[Bollob\'as~\cite{B1}]
If \(H\) is a fixed graph then
\[ \lim_{n\rightarrow\infty}\Pr\left(G_{n,M}\supseteq H\right)
   =\left\{
   \begin{array}{ll}
   1, &
   \mbox{ if $M=\omega \left(n^{2-1/m(H)}\right)$;} \\
   0, &
   \mbox{ if $M=o \left(n^{2-1/m(H)}\right)$.}
   \end{array}\right. \]
\end{theorem}
The proof of the above theorem is via the second moment method.
Let $X_H$ denote the number of copies of $H$ in $G_{n,M}$. One can conclude
something slightly stronger than stated:
\begin{equation}\label{randclique}
\mbox{If }M=\omega \left(n^{2-1/m(H)}\right)\mbox{ then \whp\ }
X_H\to\infty.
\end{equation}
Note that $m(K_r)=(r-1)/2$.  Furthermore, note
that
if $d\geq\frac{r-2}{r-1}$ then, by the
classic theorem of Tur\'an~\cite{DI}, only a balanced
$(r-1)$-partite graph fails to contain a complete $K_{r}$.  So,
we are only interested in \(d < \frac{r-2}{r-1} \).  Our main
result here is the following.
\begin{theorem}
   Let $r>r_0\geq 2$ be integers.  Let $d\in\left(\frac{r_0-2}{r_0-1},
   \frac{r_0-1}{r_0}\right]$ be a fixed constant.
   \begin{enumerate}
   \item If $H \in \cG(n,d)$ and $m=\omega\left(n^{2-2/\left(\lceil
   r/r_0\rceil-1\right)}\right)$ then $G_{H,m}$ contains a $K_r$
   \whp.  \label{KrSuf}
   \item There exists a graph $H_0 \in \cG(n,d)$
   such that if
   $m=o\left(n^{2-2/\left(\lceil r/r_0\rceil-1\right)}\right)$, then
   \whp\  $G_{H_0,m}$ fails to contain a $K_r$.  \label{KrNec}
   \end{enumerate}\label{Kr}
\end{theorem}
Note that if \( r \le 2 r_0 \) then only \( \omega(1) \) random
edges are needed to create a \(K_r\).

For \( d \le 1/2\) the minimum value of the
chromatic number of \( G_{H,m} \) taken over
\( H \in \cG(n,d) \) is largely
determined by \( \chi( G_{n,m} ) \)
(the latter is well-studied -
see, for example, \cite{F1} and \cite{L1}).
This observation follows from the simple bound
\( \chi( G_{H,m} ) \ge \chi( G_{n,m}) \) and the fact
that
\( \chi( G_{K_{n/2,n/2},m}) \)
is at most twice
\( \chi( G_{n/2,m} ) \).
In one situation
Theorem~\ref{Kr} strengthens this lower bound: if
$m = \omega(1)$ then \whp\ \(G_{H,m}\) contains a $K_4$ and
$ \chi(G_{H,m}) \ge 4$.

We now turn to the diameter of \(G_{H,m}\), which we denote
$\diam(G_{H,m})$.
If $d\geq 1/2$ then
$\diam(H) =2 $, and all remaining edges are required to
achieve diameter 1.  We first show that only a small set of random
edges is required to get \( \diam(G_{H,m}) \le 5 \).
\begin{theorem}
   If $H\in\cG(n,d)$ and
   $m=\omega(1)$ then \whp\ $\diam(G_{H,m})\leq 5$. \label{lbdiam5}
\end{theorem}
\noindent
Note that the number of edges required here is independent of \(d\).
By way of comparison, we note that Bollob\'as and Chung showed that
\whp\ the diameter of a graph consisting of a random matching added
to an \(n\)-cycle is about \( \log_2n \) \cite{BC}.
Next we show that in order to get
\( \diam(G_{H,m}) < 5 \) roughly \( \log n \) random edges are
required.
\begin{theorem} \
   \begin{enumerate}
   \item  Let \( H \) be a graph in $\cG(n,d)$.
   If $ m = \frac{ 1-d}{d^2} \log n + \omega(1)$ then
   \whp\ $\diam(G_{H,m})\leq 3$.
   \label{diam3Suf}
   \item If $d<1/2$ then there exists a graph $H_0 \in \cG(n,d)$
   such that if $m = \frac{ \log n }{ -2 \log(1-2d) }  - \omega(1) $
   then \whp\ $\diam(G_{H_0,m})\geq 5$. \label{diam3Nec}
   \end{enumerate}
\label{diam3}
\end{theorem}
\noindent
Finally, we prove that about \( n \log n \) random edges are required to
get \( \diam(G) \le 2 \).
\begin{theorem}\
   \begin{enumerate}
   \item
   Let \( H \) be a graph in $\cG(n,d)$.
   If $ m = \frac{1-d}{d} n \log n + \omega(n)$ then
   \whp\ $\diam(G_{H,m})\leq 2$.
   \label{diam2Suf}
   \item  There exists a graph $H_0 \in\cG(n,1/2)$
   such that if $m = \frac{1}{2} n \log n - \omega(n) $
   then \whp\ $\diam(G_{H_0,m})\geq 3$. \label{diam2Nec}
   \end{enumerate}
\label{diam2}
\end{theorem}
\noindent
Theorems~\ref{lbdiam5},~\ref{diam3},~and~\ref{diam2}
establish intervals for \(m\) in which the worst
case diameter of \( G_{H,m} \) is 5, 3 and 2.
Note that there is, in a sense, a
jump in the worst case diameter from 5 to 3.

The final property we consider is vertex connectivity.  First
note that if $d>1/2$ then
\( H \in \cG(n,d) \) is at least $(2\lceil dn\rceil-n+2)$-connected.
This can be seen by the fact that, by
removing any set of size $2\lceil
dn\rceil-n$, the resulting graph has minimum degree at least half the
number of vertices.  This means that the resulting graph has a
Hamilton cycle and is, thus, at least 2-connected.
Since there are graphs $H$ which are disconnected but have minimum
degree $dn$ ($d<1/2$), we focus on the number of
random edges required to make $G_{H,m}$ $k$-connected for $k\le cn$, for
some constant $c=c(d)>0$.
\begin{theorem}
   $\,$
   \begin{enumerate} \item Let \(H \in \cG(n,d)\). If $k=O(1)$ and
   $m=\omega(1)$ then $G_{H,m}$ is
   $k$-connected, \whp. If $\omega(1)\le k\le d^2n/32$ and $m=640k/d^2$
   then $G_{H,m}$ is $k$-connected, \whp. \label{connSuf}
   \item If \( d < 1/2 \) then there exists
   an $H_0 \in \cG(n,d)$ such that \whp\ $G_{H_0,m}$ fails
   to be $k$-connected for all \(k\) such that
   $m\le \frac{k}{2}\lfloor\frac{n}{dn+1}\rfloor$. \label{connNec}
   \end{enumerate}
   \label{connthm}
\end{theorem}

The remainder of the paper is organized as follows.  In the
next section we establish some preliminaries that will be
used in the proofs.  Theorem~\ref{Kr} is proved in Section~3.  In
Section~4 we prove Theorems~\ref{lbdiam5},~\ref{diam3}~and~\ref{diam2}.
Finally, we prove Theorem~\ref{connthm} in Section~5.

\begin{remark}\label{rem1}
All of the properties considered here are monotone increasing.
We therefore have an alternative way of adding random edges.
We can take each non-edge of $H$ and add it with probability $p$, creating the random graph
$G_{H,p}$ where $\brac{\binom{n}{2}-|E(H)|}p\leq m$. It then follows, as in Bollob\'as \cite{B0},
Theorem 2.2, that for a monotone increasing property $\cP$,
$G_{H,p}\in\cP$ \whp\ implies $G_{H,m}\in\cP$ \whp.
This observation can sometimes simplify our calculations.
\end{remark}

\section{Preliminaries}

We will
use the powerful Regularity Lemma of Szemer\'edi, which was formally
introduced in ~\cite{Sz}.  A survey of
applications of the Regularity Lemma is given in
in ~\cite{KS}.  In order to state the Lemma, we need a
definition.  For two
disjoint vertex sets $A$ and $B$, let $e(A,B)$ denote the number of
edges with one vertex in $A$ and the other in $B$.  Also, let
$d(A,B)=\frac{e(A,B)}{|A||B|}$.
\begin{definition} [Regularity condition]
Let $\epsilon>0$.  Given a graph $G$ and two disjoint vertex sets
$A\subset V$, $B\subset V$, we say that the pair $(A,B)$ is
$\epsilon$-regular if for every $X\subset A$ and $Y\subset B$
satisfying
\[ |X|>\epsilon |A|\qquad\mbox{and}\qquad |Y|>\epsilon |B| \]
we have
\[ \left|d(X,Y)-d(A,B)\right|<\epsilon . \]
\end{definition}

Now we can state the Regularity Lemma itself.  We will not need the
classical form but instead a modification, known as the Degree Form.
\begin{theorem}[Regularity Lemma, Degree Form]
For every $\epsilon>0$ there is an $M=M(\epsilon)$ such that if
$G=(V,E)$ is any graph and $\delta\in [0,1]$ is any real number, then
there is a partition of the vertex-set $V$ into $\ell+1$ clusters
$V_0,V_1,\ldots,V_{\ell}$, and there is a spanning subgraph $G'\subset G$ with
the following properties:
\begin{itemize}
\item $\ell\leq M$,
\item $|V_0|\leq\epsilon |V|$,
\item all clusters $V_i$, $i\geq 1$, are of the same size
$L\leq\lceil\epsilon |V|\rceil$,
\item $\deg_{G'}(v)>\deg_G(v)-(\delta+\epsilon)|V|$, for all $v\in V$,
\item $e(G'[V_i])=0$, for all $i\geq 1$,
\item {\bf all} pairs $G'(V_i,V_j)$, ($1\leq i<j\leq \ell$) are
$\epsilon$-regular, each with a density either $0$ or greater than
$\delta$.
\end{itemize}
\label{reglem}
\end{theorem}

The sets $V_i$ ($1\leq i\leq\ell$) are defined to be {\it clusters}.
In addition, we define the {\it reduced graph}.  This is the graph,
usually denoted $G_r$, which has $\{V_1,\ldots,V_{\ell}\}$ as its
vertex set and $V_i\sim_{G_r} V_j$ if and only if the pair $(V_i,V_j)$
is $\epsilon$-regular with density at least $\delta$ in $G'$.

In papers utilizing the Regularity Lemma, it is customary to use the
notation ``$\ll$''.  The expression $\epsilon\ll\delta$ means that
$\epsilon$ is ``small enough, relative'' to $\delta$.

As we have discussed, all the properties we consider are
monotone; that is, if $G$ has the property in question,
then all graphs that have $G$ as a subgraph will have
the property as well.  At several times
throughout this paper, we will want to add edges to a subgraph $H'$
instead of the original graph $H$; often, the subgraph given by
Theorem~\ref{reglem}.  Let $\cal P$ be a monotone property of graphs.
A simple coupling argument gives the intuitive result:
\begin{proposition}
   Let $H'\subset H$.  Let $G'$ be the graph formed by adding $m$
   random edges to $H'$ and $G$ be the graph formed by adding $m$
   random edges to $H$. Then
\[ \Pr\left\{G'\mbox{ has property }{\cal P}\right\}
   \leq\Pr\left\{G\mbox{ has property }{\cal P}\right\}\ . \]
\label{subprop}
\end{proposition}

%
%
%

\section{Complete subgraphs}


\begin{proof}[Proof of Theorem~\ref{Kr}]
Let us first prove~\ref{KrNec}.  Let \( H_0 \) be a
complete \( r_0 \)-partite graph
with nearly equal parts.
Note that if \( r \le 2 r_0 \) then we are adding no random
edges, so we henceforth assume \( r > 2r_0 \).  Let
\(A_1 , \dots, A_{r_0} \) be the parts of the
\(r_0\)-partition of \( V(H_0) \).  By the choice of \(m\) and
Theorem~\ref{randclique}, {\bf whp} the induced graph \( G[A_i] \)
contains no \( K_{ \lceil r/r_0 \rceil } \) for
\( i =1 , \dots ,r_0 \).  If this is the case, then the
largest clique in
\(G\) has size
\[ r_0  \left( \lceil r/r_0 \rceil -1 \right) < r. \]

%
%
%
%
%
%
%
%
%

To prove part~\ref{KrSuf} we apply the degree form of the Regularity
Lemma.
Choose $\epsilon$ and $\delta$ so that $\epsilon\ll\delta\ll d, 1/r$
and apply Theorem~\ref{reglem} to $H$ with those values, giving us the
subgraph $H'$.  The reduced graph has minimum degree
\[ (d-\delta-2\epsilon)\ell>\frac{r_0-2}{r_0-1}\ell \]
for \( \delta, \epsilon \) sufficiently small.
Therefore, it follows from Tur\'an's theorem that there
is a \(K_{r_0} \) in the reduced graph.  In other words, there
exists a collection of clusters \( A_1, \dots, A_{r_0} \subseteq
V(H^\prime) \)  such
that each pair of these sets gives an \( \epsilon \)-regular pair
with density at least \( \delta \).

%
%

We use a fact from ~\cite{KS}.  For an arbitrary graph \(H\) and
\( v \in V(H) \) let \( N(v) \) be the
set of vertices adjacent to \(v\) in \(H\).
\begin{proposition}[Intersection Property]
Let \( (A,B) \) be an \(\e\)-regular pair with density
$\d $.  If \(Y \subseteq B\),
and
$(\d -\epsilon)^{k-1} |Y| >\epsilon |B| $ then
\[ \#\left\{\left(x_1,x_2,\ldots,x_{k}\right) : x_i\in A
   \left| \left( \bigcap_{i=1}^{k}N(x_i) \right) \cap Y \right| \leq
   ( \d -\epsilon)^{k}|Y|\right\}\leq k\epsilon |A|^{k} \]
\label{intprop}
\end{proposition}

We add
\( \omega( n^{2-2/\left(\lceil r/r_0\rceil-1\right)} ) \)
random edges to \( H^\prime \) in a sequence of \( r_0 \)
phases.
$\varphi=\varphi(n)$ will be a function that
goes to infinity as
$n\rightarrow\infty$ and we let \( R = R_1 \cup
\dots \cup R_{r_0} \) where
\( |R_i| = \varphi^{2} n^{2-2/\left(\lceil r/r_0\rceil-1\right)} \)
for \( i =1, \dots , r_0 \). (Recall that \(R\) is the set of random
edges.)

It follows from (\ref{randclique})
that \whp\ \( R_1 \cap \binom{A_1}{2} \)
contains at least \( \varphi \) copies
of
\(K_{ \lceil r/ r_0 \rceil} \).  It follows from
the Intersection Property that among these \whp\ there
exists a copy of \( K_{ \lceil r/ r_0 \rceil} \) such that
intersection of the neighborhoods of the vertices in
this \( K_{ \lceil r/ r_0 \rceil} \) is large
in \( A_2, \dots, A_{r_0} \); to be precise, \whp\ there
 exists \( X_1 \subseteq A_1 \) such that
\begin{itemize}
\item[(i)] \( |X_1| =  \lceil r/ r_0 \rceil \),
\item[(ii)] \( \binom{X_1}{2} \subseteq R_1 \), and
\item[(iii)] For \( i =2, \dots, r_0 \) we have
\[ \left| \left( \bigcap_{x \in X_1} N(x) \right)
\cap A_i \right| \ge ( \d - \e)^ {  \lceil r/ r_0 \rceil} |A_i| \]
\end{itemize}
For \( i =2 , \dots, r_0 \) let
\[ A_{i,1} =  \left( \bigcap_{x \in X_1} N(x) \right)
\cap A_i. \]

In the phase \(j\), \(j=2, \dots, r_0\), we note that it follows from
Theorem~\ref{randclique} that \whp\ \( R_j \cap \binom{ A_{j,j-1}}{2} \)
contains at least \( \varphi \) copies of
\(K_{ \lceil r/ r_0 \rceil} \).  It then follows from an
application of the Intersection Property
that \whp\ there
 exists \( X_j \subseteq A_{j,j-1} \) such that
\begin{itemize}
\item[(i)] \( |X_j| =  \lceil r/ r_0 \rceil \),
\item[(ii)] \( \binom{X_j}{2} \subseteq R_j \), and
\item[(iii)] For \( i =j+1, \dots, r_0 \) we have
\[ \left| \left( \bigcap_{x \in X_j} N(x) \right)
\cap A_{i,j-1} \right| \ge ( \d - \e)^ {  \lceil r/ r_0 \rceil} |A_{i,j-1}| \]
\end{itemize}
If \( j < r_0 \) then, for \( i =j+1 , \dots, r_0 \), let
\[ A_{i,j} =  \left( \bigcap_{x \in X_j} N(x) \right)
\cap A_{i,j-1}. \]
If \(j=r_0 \) then note that \( X_1 \cup \dots \cup X_{  r_0 } \) is
a \( K_r \) in \(G_{H,m} \).

Note that this argument requires
\[ \frac{ k \e}{ (\d - \e)^{r_0k^2 }} < 1  \hskip5mm \text{ and }
\hskip5mm
(\d - \e)^{r_0k} > \e \]
where \( k = \lceil r/r_0 \rceil \).
Of course, we can attain these inequalities
by choosing \( \e \) sufficiently small.

\end{proof}

\begin{remark} The proof of Theorem~\ref{Kr}~(a)
goes through if the condition \( H \in {\mathcal G}(n,d) \) is
replaced with the condition
\[  \sum_{v \in V(H)} deg(v) \ge 2d |V(H)|, \]
i.e. the average degree of \(H\) is at least \(d\).  Just as a
minimum degree condition on \(H\) implies a minimum degree condition on
\( H' \) that in turns implies a minimum degree condition on
the reduced graph, a lower bound on the average degree of \( H\)
implies a lower bound on the average degree of \(H'\) that in turn gives
a lower bound on the average degree of the reduced graph.  Such a
bound on the average degree
is all that is required to establish the
existence of a \(K_{ \lceil r/r_0 \rceil} \) in the
reduced graph (by Tur\'an's Theorem).
\end{remark}

\section{Diameter}

\begin{proof}[Proof of Theorem~\ref{lbdiam5}]

We apply the Regularity Lemma to $H$ with parameters
$ \epsilon \ll \d \ll d $.  This gives us the reduced graph $H_r$ as
well as the subgraph $H'$.  We will work with $H'$ instead of $H$ by
invoking Proposition~\ref{subprop}.

For each cluster $V$ let $V'$
be a fixed cluster
such that $(V,V')$ is, in $H'$, $\epsilon$-regular with density
at least $\d$ (i.e. $V\sim V'$ in $H_r$).  Note that
for every vertex \(x\) there exists a cluster $V_x$ in which $x$ has large
degree in \(H'\) (note that \(x \not\in V_x\)).  Indeed, using the conditions
of Theorem~\ref{reglem}, the minimum
degree condition and the fact that $|V_0| \le \epsilon n $,
we see that $x$ is adjacent to at least
\( dn - (\d+\epsilon)n - \epsilon n
   = (d-\d-2\epsilon)n \)
vertices in $\bigcup_{i=1}^{\ell}V_i$.  So, there is a cluster $V_x$
such that
$x$ is adjacent to at least
\begin{equation}
\label{eq:linkdeg}
 \frac{(d-\d-2\epsilon)n}{\ell}\geq (d-\d-2\epsilon)L
\end{equation}
vertices in $V_x$.

For arbitrary vertices $u$ and $w$
we find a path of length at most 5 from \(u\)
to \(w\) using the clusters \( V_u, V_w, V_u'\) and \( V_w'\).
Since $V_u\sim_{H_r}V_u'$, at least
$(1-\epsilon)L$ vertices in $V_u'$
are at a distance two from $u$.  The same holds for $w$.
It follows, for example, that
$V_u' = V_w'$ implies
$\dist(u,w)\leq 4$.  It is also
easy to see that if
\( V_u' \sim_{H_r} V_w' \) then $\dist(u,w)\leq 5$.
So, we may assume that the clusters \( V_u, V_w, V_u' \) and \( V_w'\)
are distinct and
\[ V_u\sim_{H_r}V_u'\qquad V_u'\not\sim_{H_r}V_w'\qquad
   V_w'\sim_{H_r} V_w. \]
We use the random edges between \( V_u' \) and \( V_w'\)
to complete the path from \(u\) to \(w\).

%

For a fixed constant \(k\) and fixed clusters
\(V_1, V_2, V_3, V_4 \) such that
\begin{equation}
\label{eq:adjs}
V_1 \sim_{H_r} V_2 \hskip5mm \text{ and } \hskip5mm
V_3 \sim_{H_r} V_4
\end{equation}
we define a {\em $k$-link} of the quadruple
\( V_1, V_2, V_3, V_4 \) to be a collection of \(k^2 \) edges
\[ \left\{ e_{i,j} = \{ a_{i,j}, b_{i,j} \}
: 1 \le i,j \le k \right\} \]
such that
\begin{itemize}
\item[(i)] \( a_{i,j} \in V_2 \) for \( 1 \le i,j \le k \),
\item[(ii)] \( b_{i,j} \in V_3 \) for \( 1 \le i,j \le k \),
\item[(iii)] For \( i =1, \dots, k\) we have
\[ \left| \bigcup_{j=1}^k N(a_{i,j}) \cap V_1 \right|
> L - ( d - \d-  2\e)L \]
\item[(iv)] For every function \( \sigma : [k] \to [k] \) we
have
\[ \left| \bigcup_{i=1}^k N(b_{i,\sigma(i)}) \cap V_4 \right|
> L - ( d - \d - 2\e)L \]
\end{itemize}
We show below that there exists a \(k\) such that \whp\ \(R\) contains
a \(k\)-link for every quadruple of clusters satisfying (\ref{eq:adjs}).

Before proving the existence of these \(k\)-links, we show that
if \( R \) contains a \(k\)-link of \( V_u, V_u', V_w', V_w \)
then there is a path of length 5 from \(u\) to \(v\) in \( G_{H,m} \).
Let \( X = N(u) \cap V_u \) and \(Y = N(w) \cap V_w \).  It follows
from (\ref{eq:linkdeg}) and (iii) that for \( i = 1, \dots k \) there exists
\( \sigma(i) \in [k] \) such that \( a_{i,\sigma(i)} \) has a neighbor
\(x_i \in  X \).  It then follows from (\ref{eq:linkdeg}) and (iv)
that there exists \(j \in [k ] \) such that \( b_{j, \sigma(j)} \)
has a neighbor \(y \in Y\).  The following sequence of vertices
is a path in \( G_{H,m}\): \(u,x_j, a_{j,\sigma(j)}, b_{j,\sigma(j)},y,w\).

It remains to show that there exists a \(k\) such
that \whp\ \(R\) contains
a \(k\)-link for every quadruple of clusters
satisfying (\ref{eq:adjs}).  Since the number of vertices in
the reduced graph \( H_r \) is bounded, it suffices to show that
there exists a constant \(k\) such that
\whp\ a fixed quadruple \( V_1, V_2, V_3 ,V_4 \)
satisfying (\ref{eq:adjs}) has a \(k\)-link.
We use the following simple
fact regarding \(\e\)-regular pairs.
\begin{proposition}[Union Property]
Let $(A,B)$ be an $\epsilon$-regular
pair with density at least $\delta$.
If $k\geq 1$ and $(1-\d+\epsilon)^{k-1}\geq\epsilon$ then
\[ \#\left\{\left(x_1,x_2,\ldots,x_{k}\right) : x_i\in A,
   \left|\bigcup_{i=1}^{k}N(x_i)\right| \leq
   \left(1-(1-\d+\epsilon)^{k}\right)|B|\right\}\leq
   k\epsilon |A|^{k} \]
\label{unionprop}
\end{proposition}

\begin{proof}
We go by induction on $k$.  For $k=1$ this statement
follows directly from
the definition of $\epsilon$-regularity.
For \( k \ge 2 \), let
\( {\mathcal X} \) be
the collection of $(k-1)$-tuples \( x_1, \dots ,x_{k-1} \)
of elements of
\(A\) with the property that
\[ \left| \bigcup_{j=1}^{k-1} N(x_j) \right| \le
\left(1-(1-\d+\epsilon)^{k-1}\right)|B| .\]  By induction,
\( | {\mathcal X}| \le
(k-1)\epsilon |A|^{k-1} \).
For a fixed \( (k-1)\)-tuple \( v_1,\ldots, v_{k-1} \) not in
\( {\mathcal X} \),
the number of vertices that are adjacent to fewer than
$(\d-\epsilon)(1-\d+\epsilon)^{k-1}|B|$ vertices
among a set of $(1-\d+\epsilon)^{k-1}|B|$ vertices that
contain the complement of that union is at most $\epsilon |A|$, by the
definition of $\epsilon$-regularity.
Thus, the number of bad $k$-tuples is at most
\[
|{\mathcal X}| \cdot |A| + |A|^{k-1}\cdot \epsilon|A|
\le (k-1)\epsilon |A|^{k-1}\cdot |A|+|A|^{k-1}\cdot \epsilon|A|
= k\epsilon |A|^k . \]
\end{proof}
Let $k$ be a constant such that
\begin{gather*}
1-(1-\delta+\epsilon)^k\geq 1-(d-\delta-2\epsilon) \\
(1-\delta+\epsilon)^{k-1}\geq\epsilon, \; \text{ and}  \\
(k + k^{k})k\epsilon<1.
\end{gather*}
Clearly, such a \(k\) exists if we choose the constants
\( \epsilon \) and \( \delta \) appropriately.

We are now ready to consider the random edges.   Clearly,
the probability that
there are fewer than, say, \( \sqrt{m} \)
edges \( e  \in R\) that intersect both \( V_2\) and
\( V_3\) is \( o(1) \).  Conditioning on this event, we may assume
that there is a collection of
\( \sqrt{m} = \omega(1) \) edges \(R'\)
between clusters \( V_2 \) and \( V_3 \) chosen uniformly at
random from the
collection of all such edges.  It follows from the Union Property
that the probability that a
fixed set of \( k^2 \) such edges does not form a \( k\)-link
is at most \( (k + k^k) k\e \).  By our choice of \(k, \e\) and \(\d\),
this probability is less than 1.  Partitioning \( R' \)
into \( \omega(1) \) collections of \( k^2 \) edges, we see that
\whp\ one of these parts forms a \(k\)-link.

\end{proof}

\begin{proof}[Proof of Theorem~\ref{diam3}]
The proof of~\ref{diam3Suf} is a straightforward computation
showing that every pair of disjoint
neighborhoods must have an edge between
them.
Following Remark \ref{rem1} we consider
$G_{H,p}$ with $p=\frac{2(\log n+\om(1))}{d^2n^2}$.
\begin{align*}
   & \Pr\left\{\exists u,v \in \binom{[n]}{2}
          : R \cap \left(N(u) \times N(v)\right)=\emptyset
            \text{ and } N(u) \cap N(v) = \emptyset \right\} \\
 & \hskip9cm \le   \binom{n}{2}(1-p)^{d^2n^2} \\
 & \hskip9cm \le
   \binom{n}{2} e^{-d^2n^2p} \\
 & \hskip9cm =  o(1).
\end{align*}

In order to prove the existence of the
graph \( H_0 \) for~\ref{diam3Nec} we consider the graph \(H\)
which is the disjoint union of \( G_{ \lfloor n/2 \rfloor, p} \) and
\( G_{ \lceil n/2 \rceil, p} \) where \( p = 2d + n^{-1/3}\).
In other words, we form a partition \([n] = S \cup T\) where
\( |S|= \lfloor n/2 \rfloor \) and \(|T|=  \lceil n/2 \rceil\), place no edges
between the two parts, and place each edge that lies within
one of the parts, independently, with probability
\(p\).

Let \(R\) be an arbitrary set of \(m\) edges
on vertex set
\( [n] = S \cup T \).  Let \(S'\) and \(T'\) be the vertices
in \(S\) and \(T\), respectively, that are in edges in \(R\).
We begin by showing
that for this fixed set \(R\) \whp\
the random graph \( H \) has the following property:
There
exist \(u \in S \) and \(v \in T \) such that
\( (\{u\}\cup N(u)) \cap S' = \emptyset\) and
\( (\{v\} \cup N(v)) \cap T' = \emptyset \).
Let \(r = | S'| \).
The
probability that \(N(u) \cap S'\) is nonempty
for {\em every} vertex
in \(S \setminus S' \) is at most
\begin{equation*}
\begin{split}
(1 - (1-p)^r)^{n/2-r} & = \left( 1 - e^{r \log(1-p)} \right)^{n/2-r} \\
& \le \exp \left\{ - e^{r \log(1-p) }( n/2-r) \right\} \\
& \le \exp \left\{ - e^{2m \log(1-p) }( n/3) \right \} \\
& = o(1)\ .
\end{split}
\end{equation*}
It follows from this observation (and an application of the
Chernoff bound to the degrees in \(H\))
that there exists a fixed
graph \(H\)
such that
\begin{itemize}
\item[(i)] The minimum degree of \(H\) is greater than \(dn\).
\item[(ii)] \whp\
(with respect to the random choice of \(R\), now)
there exist \(u \in S\) and \( v \in T \) such
that no vertex in the set \( \{u,v\} \cup N(u) \cup N(v) \)
intersects an edge in \(R\).
\end{itemize}
The distance in \( G_{H,m} \) between the
vertices \(u\) and \(v\) given by (ii) is at least 5.

\end{proof}

\begin{proof}[Proof of Theorem~\ref{diam2}]
The proof of~\ref{diam2Suf} is also a straightforward calculation.
We show that there is an edge between every vertex and
neighborhood. Following Remark \ref{rem1}, we consider $G_{H,p}$ with
$p=\frac{2(\log n+\om(1))}{dn}$.
\begin{eqnarray*}
   \Pr\left\{\exists u,v : R \cap \left( \{u\} \times N(v) \right)
=\emptyset\right\} &
   \leq & \binom{n}{2}(1-p)^{dn}\\
   & \leq &
   \binom{n}{2} e^{-np} \\
   & = &o(1).
\end{eqnarray*}

The example of $H_0$ that we use for~\ref{diam2Nec} is
simply two complete graphs, each on \( n/2 \) vertices,
with no edges between them.  Let the vertex sets for these
complete graphs be \( S\) and \(T\).  Note that if there
exists a vertex \(u \in S\) that is in no edge in \(R\) and
there exists a vertex \(v \in T \) that is in no edge in \(R\)
then the distance between \(u\) and \(v\) in \( G_{H_0,m} \) is
at least 3.

We compute the probability that \(R\) touches every
vertex in \(S\).
\begin{equation*}
\begin{split}
Pr \left\{ \forall v \in S \exists e \in R
\text{ such that } v \in e \right\} &
\le \left( 1 - ( 1 -2/n)^m \right)^{n/2} \\
& \le \left( 1 - e^{- 2m/n - 4m/n^2} \right)^{n/2} \\
& \le \exp \left\{ - \frac{n}{2} e^{- 2m/n - 4m/n^2} \right\} \\
& = o(1)\ .
\end{split}
\end{equation*}
Note that we used the fact that the event
$\{\mbox{$\exists e \in R$ such that  $v \in e$}\}$ is
negatively
correlated with events of the form
$\{\mbox{ $\forall w \in W \exists e \in R $ such that $w \in e
$}\}$, where  $W \subseteq  S\setminus\{v\}$.

%
%
%
%
%

\end{proof}

\begin{remark}
Improvements in the constants in
Theorems~\ref{diam3}\ref{diam3Nec}~and~\ref{diam2}\ref{diam2Nec}
could be obtained by noting that a significant portion of the
random edges in \(R\) fall within the two parts of the
partition (of course, one would have to replace the
complete graphs used in the proof of Theorem~\ref{diam2}\ref{diam2Nec}
with arbitrary graphs of the desired density to get this
improvement there).  This observation was not used here for the
sake of brevity.

\end{remark}

\section{Vertex connectivity}
Let \( \kappa(G) \) denote the vertex connectivity of graph \(G\).
We first prove the following lemma that may be of independent interest.
\begin{lemma}\label{part_con}
Let $H=(V,E)$ be a graph on $n$ vertices with minimum degree $k>0$. Then
there exists a partition $V=V_1\cup\ldots\cup V_t$ such that for every
$1\le i\le t$ the set $V_i$ has at least $k/8$ vertices and the induced
subgraph $H[V_i]$ is $k^2/(16n)$-connected.
\end{lemma}

\begin{proof}
Recall the following classical result on vertex connectivity.
\begin{theorem}[Mader (see~\cite{DI})]
Every graph of average degree at least $k$ has a $k/4$-connected
subgraph.
\label{mader}
\end{theorem}

Let $(C_1,\ldots,C_t)$ be a family
of disjoint subsets of $V$ with the property that each induced
subgraph $H[C_i]$ is $k/8$-connected and that, among all such
families of subsets, the set of vertices
\[ C\stackrel{\rm def}{=}\bigcup_{i=1}^t C_i \]
is maximal.
According to Theorem~\ref{mader}, $t>0$.  Also, $|C_i|\geq k/8$ for
all $i$ and thus $t\leq 8n/k$.

Let now $(V_1,\ldots,V_t)$ be a family of disjoint subsets of $V$ such
that $C_i\subseteq V_i$, the induced subgraph $H[V_i]$ is
$k^2/(16n)$-connected for all $1\le i\le t$ and that among all such
families the set of vertices
\[ U\stackrel{\rm def}{=}\bigcup_{i=1}^t V_i\]
is maximal. We claim that $U=V$. Assume to the contrary that there
exists a vertex $v\in V\setminus U$. If $|N(v)\cap V_i|\ge k^2/(16n)$
for some $i$, then adding $v$ to $V_i$ can be easily seen to keep
$H[V_i]$ $k^2/(16n)$-connected, contradicting the maximality of $U$.
Thus $v$
has less than $k^2/(16n)$ neighbors in each of the $t\le 8n/k$ sets
$V_i$, and therefore $deg_{V\setminus U}v> k-(8n/k)(k^2/(16n))=k/2$. We
conclude that the minimum degree of the induced subgraph $H[V\setminus
U]$ is at least $k/2$. Applying Theorem ~\ref{mader}, this time to
$H[V\setminus U]$, unveils a $k/8$-connected subgraph disjoint from $C$
-- a contradiction to the choice of $(C_1,\ldots,C_t)$. Hence the family
$(V_1,\ldots,V_t)$ covers indeed all the vertices of $H$ and thus forms
a required partition.
\end{proof}

We remark that the above result is optimal up to constant multiplicative
factors. To see this take
$ \lceil (n-k^2/n)/(k+1) \rceil$
disjoint cliques $C_i$ of size
$k+1$ each, add an independent set $I$ on the (at most
$k^2/n$) remaining vertices, and connect
each vertex of $I$ with roughly $k^2/n$ arbitrarily chosen
vertices of $C_i$,
$1\le i\le \lceil (n-k^2/n)/(k+1) \rceil$.
Denote the obtained graph by $H$.  Let
$K\subseteq H$ be a subgraph of $H$ containing some vertices from $I$.
If $K$ intersects two distinct cliques $C_i, C_j$, then deleting
$V(K)\cap I$ disconnects $V(K)\cap C_i$ from $V(K)\cap C_j$, and thus
the connectivity of $K$ does not exceed $|V(K)\cap I|\le |I| \le k^2/n$. If
$K$ intersects a unique clique $C_i$, then deleting all neighbors of
$v\in V(K)\cap I$ from $C_i$ disconnects $v$ from the rest of $K$,
implying $\kappa(H[K])\le deg_{C_i}v\le k^2/n$.

\begin{proof}[Proof of Theorem~\ref{connthm}]

Let us begin with part~\ref{connSuf}.  Let $H$ be a graph
with minimum degree at least $dn$.
Let $(V_1,\ldots,V_t)$ be a
partition of $V(H)$ such that $|V_i|\ge dn/8$ and $H[V_i]$ is
$(d^2n/16)$-connected, $1\le i\le t$. The existence of such a partition is
guaranteed by Lemma \ref{part_con}.
It is enough to show that the graph $G_{H,m}$ \whp\ contains a
matching of size $k$ in each bipartite graph induced by $(V_i, V_j)$.
Let $F_{ij}$ be a maximum matching between $V_i$ and $V_j$ in $H$. If
$|F_{ij}|\ge k$ we are done. Assume therefore that $|F_{ij}|<k$.
Choose a subset $A\subset V_i\setminus\bigcup_{e\in F_{ij}}e$
and a subset $B\subset V_j\setminus\bigcup_{e\in F_{ij}}e$ of cardinalities
$|A|=|B|=3dn/32$. Obviously $H$ has no edges connecting $A$ and $B$ due
to the maximality of $F_{ij}$.

Consider first the case $k=O(1)$. Then the set $R$ contains \whp\
$\omega(1)$ random edges between $A$ and $B$. If
$\omega(1)=o(\sqrt{n})$, then those random edges form \whp\ a matching
as required.

Let now $k=\omega(1)$. Due to Remark \ref{rem1} we may consider
$G_{H,p}$ with $p=\frac{1280 k}{d^2n^2}$. Then the probability that the
set $R$ of random edges does not have a matching of size $k$ between $A$
and $B$ can be estimated from above by:
$$
\sum_{i=0}^{k-1}{\binom{\frac{3dn}{32}}{i}}^2i!p^i
(1-p)^{\left(\frac{3dn}{32}-i\right)^2}
$$
(This expression  arises from first
choosing the size $i<k$ of a maximum matching $M$ between $A$ and
$B$ in $R$, then choosing the vertices of $M$ in $A$ and $B$,
then forming a pairing between
them, then requiring all matching edges to be present in $R$, and
finally requiring all
edges lying outside the vertices of the matching to be absent). We can
estimate the above expression from above by:
\begin{equation*}
\begin{split}
e^{-p\left(\frac{3dn}{32}-k\right)^2}
 \left( 1+\sum_{i=1}^{k} \left(\frac{3edn}{32i}\right)^{2i}i^ip^i \right)
& < e^{-p\left(\frac{3dn}{32}-\frac{d^2n}{32}\right)^2}
 \left(1+\sum_{i=1}^k\left[\left(\frac{3edn}{32i}\right)^2ip\right]^i
 \right)\\
& < e^{-\frac{d^2n^2p}{256}}
\left(1+\sum_{i=1}^k\left(\frac{9e^2d^2n^2p}{1024i}\right)^i\right) \\
& = e^{-5k}\left(1+\sum_{i=1}^k\left(\frac{45e^2k}{4i}\right)^i\right)
\end{split}
\end{equation*}
In the last sum each summand is easily seen to be at least twice larger
than the previous summand, and hence the above estimate is at most
$e^{-5k}2(\frac{45e^2k}{4k})^k=o(1)$.

As to part~\ref{connNec}, let $H_0$  consist of $\lfloor
\frac{n}{dn+1}\rfloor$ disjoint cliques $C_1,\ldots,C_t$ each of size at
least $dn+1$. If $H\cup R$ is $k$-connected, then each $C_i$ is incident
to at least $k$ edges, implying $|R|\ge
\frac{kt}{2}=\frac{k}{2}\lfloor\frac{n}{dn+1}\rfloor$.
\end{proof}

\noindent{\bf Acknowledgment:}  The authors thank the anonymous
referees for their helpful comments.

%
%
%

\end{document}